\documentclass{amsart}
\usepackage{amsmath,amssymb,amsthm,cite,graphicx,setspace,verbatim,fullpage,bbm,nicefrac}
\usepackage[usenames]{xcolor}
\usepackage[margin=1.0in]{geometry}
\newtheorem{lemma}{Lemma}
\newtheorem{theorem}{Theorem}

\newcommand{\Z}{\mathbb{Z}}
\newcommand{\R}{\mathbb{R}}
\newcommand{\p}{\mathbb{P}}
\newcommand{\lilOh}[1]{\ensuremath{\mathit{o}\!\left( #1 \right)}}
\newcommand{\bigTheta}[1]{\ensuremath{\Theta\!\left(#1\right)}}
\newcommand{\eps}{\epsilon}
\newcommand{\La}{\mathbb{L}}
\newcommand{\e}[1]{\mathbb{E}\left[#1\right]}
\newcommand{\bigOh}[1]{\ensuremath{\mathcal{O}\!\left( #1 \right)}}
\DeclareMathOperator{\DIAM}{diam}

\newcommand{\diam}{\DIAM }

\newcommand{\ba}{\backslash}

\begin{document}

\title{Connectivity and giant component in random distance graphs}
\date{\today}
\author{Joshua Flynn}
\email{jflynn40@uw.edu}

\author{Briana Oshiro}
\email{bsoshiro@uw.edu}

\author{Mary Radcliffe}
\email{mradclif@andrew.cmu.edu}

\address{Department of Mathematics, University of Washington,
Seattle, WA 98195}

\begin{abstract}
Various different random graph models have been proposed in which the vertices of the graph are seen as members of a metric space, and edges between vertices are determined as a function of the distance between the corresponding metric space elements. We here propose a model $G=G(X, f)$, in which $(X, d)$ is a metric space, $V(G)=X$, and $\p(u\sim v) = f(d(u, v))$, where $f$ is a decreasing function on the set of possible distances in $X$. We consider the case that $X$ is the $n\times n \times \dots\times n$ integer lattice in dimension $r$, with $d$ the $\ell_1$ metric, and $f(d) = \frac{1}{n^\beta d}$, and determine a threshold for the emergence of the giant component and connectivity in this model. We compare this model with a traditional Waxman graph. Further, we discuss expected degrees of nodes in detail for dimension 2. 
 \end{abstract}
\maketitle

\section{Introduction}

The study of random graphs dates back to the work of Erd\H{o}s and Renyi in the late 1950s \cite{erdHos1960evolution,erdHos1959random}.
In particular, the transition of a random graph as a composition of mostly small components to one with a ``giant component'' to a connected graph has been studied extensively. These structural changes are known as phase transitions, and the two-step process above is sometimes referred to as the ``double-jump'' in a random graph. Beginning with the revolutionary work of Erd\H{o}s and R\'{e}nyi, phase transitions have been studied in a multitude of different settings (see, for example, \cite{bollobas2007phase, bollobas2012simple,ding2011anatomy,janson1993birth,janson2012phase, spencer2010giant}, among many others).  
Of the first pieces of work in the subject, Erd\H{os} and R\'{e}nyi published an analysis of component sizes and phase transitions in random graphs 
\cite{erdHos1960evolution,erdHos1959random}.

The first of the commonly studied random graphs is known as the Erd\H{o}s-R\'{e}nyi model. In this model, the number of vertices is denoted $n$ and the probability that two vertices are adjacent is denoted $p$, where each edge is included independently. Herein we use the notation $G(n, p)$ for this graph. Phase transition has been studied extensively in $G(n, p)$. For probability $np \lesssim c<1 $, this model contains only small components of size at most $\bigOh{\log n}$ asymptotically almost surely. For $ 1<c\lesssim np \lesssim \log n $,  there is a giant component of size on the order of $n$ asymptotically almost surely, and for $np \gtrsim \log n$, the graph is connected asymptotically almost surely (see, for example \cite{AlonandSpencer} for an analysis of component sizes in $G(n, p)$).

In this work, we study random graphs for which the vertices are embedded in a metric space, and edges are chosen based upon the distance between these vertices. Examples of graphs of this type in the literature are abound, such as random geometric graphs (see, for example, \cite{penrose2003random,balister2008percolation,mahadev1995threshold}), geographical threshold graphs (see, for example, \cite{bradonjic2007wireless,masuda2005geographical,bradonjic2007giant}), the Kleinberg small world model (see, for example, \cite{garfield1979its,kleinberg2000navigation}), Waxman models (see, for example, \cite{waxman1988routing,van2001paths,naldi2005connectivity}), among others. See \cite{avin2008distance} for a description of some of these types of graphs.

The ``randomness'' in such graphs is generally presented in one of two ways: either vertices are chosen randomly or the edges are chosen randomly. For random geometric graphs and geographical threshold graphs, the vertices are chosen randomly from an underlying metric space, and then a rule is devised to determine their adjacency; typically the adjacency is deterministic once the vertex set has been chosen. Often the metric space in question here is $\R^d$, although that is not strictly necessary. On the other hand, for the Kleinberg small world model, the vertices are fixed in the metric space, but the presence of edges is chosen randomly as a decreasing function of the distance between nodes. The Waxman model, seemingly uniquely among graphs of this type, chooses both the position of the vertices in the metric space and the edges randomly. 

In this work, we propose a graph model similar to a Kleinberg model or Waxman model. We consider a sequence of random graphs defined as follows. First, fix some metric space $(X, d)$, and take $X_1\subset X_2\subset X_3\subset\dots\subset X$ to be a sequence of sub-metric spaces of $X$, in which $|X_j|$ is finite for all $j$. For each $j$, let $f_j:[0, \diam{X_j})\to [0, 1]$ be a decreasing function. The graph $G(X_j, f_j)$ is defined by $V(G(X_j, f_j))=X_j$ and for any $u, v\in X_j$, $\p(u\sim v) = f_j(d(u, v))$. We refer to such a model in this work as a {\it random distance graph}.

To distinguish this model from the existing literature, we include an analysis of connectivity and other structures in the traditional Waxman model in Section \ref{S:Waxman}. Given an metric space $(X, d)$, together with a probability distribution $\mu$ over $X$, let $W(X, n, f)$ denote the traditional Waxman model over $X$ with $n$ vertices, wherein each vertex is embedded randomly in $X$ according to $\mu$, and the probability that two vertices are adjacent is given by $\p(u\sim v)=f(d(u, v))$. We prove the following.

\begin{theorem}\label{fixedms}
Let $X$ be a connected metric space with finite diameter $d_X$ and let $\mu$ be a probability distribution over $X$. Let $G=W(X, n, f_n)$. If there exists $\eps>0$ such that the functions $f_n$ satisfy the condition
\[
\frac{n}{\log n} f_n(d_X) > 1+\eps
\]
for $n$ sufficiently large, then the graph $G$ is connected a.a.s.
\end{theorem}

We note that in practice, Waxman models are typically used with $X$ a finite volume subset of $\R^k$, and $f=f_n= \alpha e^{-\frac{d}{\beta}}$ with $\alpha, \beta \in (0,1]$. Note that an immediate corollary to the above theorem is that all traditional Waxman graphs are connected asymptotically almost surely.

In contrast, we have the following theorems regarding connectivity in random distance graphs. 

\begin{theorem}\label{T:conn}
Let $G_n=G(X_n,f_n)$. If there exists $\epsilon>0$ such that \[f_n(\diam X_n )>\frac{(1+\epsilon) \ln |X_n|}{|X_n|},\] then $G_n$ is connected a.a.s..
\end{theorem}

\begin{theorem}\label{T:disconn}
Let $(X_n)$ be a sequence of nested finite metric spaces with metric $\rho$, and let $\rho_n=\diam X_n$. Let $G_n=G(X_n,f_n)$. For each $n, d>0$, define \[a_n(d) = \sup_{v\in X_n}\left|\left\{u\in X_n\ | \ \rho(u, v)=d\right\}\right|.\] If there exists $\alpha>0$ such that
\[
f_n(d)\leq\frac{|X_n|^{-\alpha}}{a_n(d)\rho_n}
\]
for all $d>0$ and $n$ sufficiently large, then there exists $\epsilon>0$ such that with probability at least $1-n^{-\epsilon}$, $G_n$ has $|X_n|(1-o(1))$ isolated vertices.
\end{theorem}

We note that there is bit of difference in the size of $f_n$ in Theorems \ref{T:conn} and \ref{T:disconn}. However, we have a precise threshold in the case that $(X, d)$ is the $r$-dimensional integer lattice, under the $\ell_1$ metric, and 
\[ X_n = \{(a_1, a_2, \dots, a_n)\in X\ | \ 0\leq a_i\leq n-1 \hbox{ for all }i\};\]that is, $X_n$ is a $n\times n\times\dots\times n$-sized subset of $\Z^r$. We typically write $L_n^r$ to denote this metric space.

Using this metric space, we can view adjacency between nodes of $X_n$ as a function of the difference between corresponding coordinates in the vector representing the node. Examples of graphs defined in a similar way include stochastic Kronecker graphs (see, for example, \cite{mahdian2007stochastic,leskovec2010kronecker,radcliffe2013connectivity}), multiplicative attribute graphs (see, for example, \cite{kim2011modeling,kim2012multiplicative}), and random dot product graphs (see, for example, \cite{scheinerman2010modeling, nickel2007random, young2007random}) . However, in these cases, the specific values of the coordinates is taken into account in determining the probability that two nodes are adjacent, whereas in this model, the only determining factor is the difference between corresponding coordinates.

In this particular case, we obtain the following result, which closes the gap between Theorems \ref{T:conn} and \ref{T:disconn}.

\begin{theorem}\label{T:main}
Let $f_n(d) = \frac{1}{n^\beta d}$, where $\beta\in \R$. Fix $r>0$, and let $G=G(L_n^r, f_n)$. Then
\begin{enumerate}
\item if $\beta<r-1$, then $G$ is connected a.a.s., and
\item  if $\beta>r-1$, then $G$ is disconnected a.a.s., and, moreover, $G$ has $n^r(1-\lilOh{1})$ isolated vertices.
\end{enumerate}
\end{theorem}

This behavior is striking in that we do not see the typical ``double-jump'' between a giant component and a connected graph when $\beta$ is constant. Instead, the graph goes from having no large component directly to being connected. It seems likely that a more nuanced approach to choosing $\beta$ as a function of $n$ may identify a double jump here.

The approach to the proof of Theorem \ref{T:main} involves an approximation of the expected degree of each vertex, obtained by first expanding the vertex set to the infinite lattice $\Z^r$, and then considering an appropriately chosen subset. We use an approximation for the expected degree, but also include a proof of the precise expected degree in dimension $2$ for comparison in Section \ref{S:degreetwo}. 

We note also that this choice of $f_n$ is designed to keep the graph sparse even as $n\to \infty$, in keeping with expectations for small world models (see, for example, \cite{humphries2008network, easley2010networks}). We shall see as a corollary that the density of the graph in fact tends to 0 as $n\to\infty$.

\section{Tools and notation}\label{S:notation}

Throughout this paper, we use standard graph theoretic notation and terminology. Our primary language is defined below; we refer the reader to \cite{bollobas1982graph} for any terminology not herein defined.

Let $G=(V, E)$ be a graph. Given a vertex $v\in V$, define $\deg_G(v)$ to be the degree of $v$ in $G$. If $G$ is understood, we shall write $\deg(v)$ for brevity. We write $u\sim v$ to indicate that $u$ is adjacent to $v$. A graph $G$ is connected if for any two vertices $u, v\in V$, there is a path between $u$ and $v$. A maximal connected subgraph of $G$ is called a connected component, or simply a component of $G$. A graph family $G_n$ is said to have a giant component if there exists a connected component containing $\bigTheta{|V(G_n)|}$ vertices of $G_n$ for each $n$.

A graph $G$ is called {vertex transitive} if for every $u, v\in V(G)$, there exists a function $\phi:V(G)\to V(G)$ such that $\phi(u)=v$, and $x\sim y$ if and only if $\phi(x)\sim \phi(y)$; that is, there is a homomorphism that sends any vertex in $G$ to any other. 

Throughout, we shall focus on graphs with a fixed vertex set and randomly generated edges, where the set of edges are mutually independent. Given a sequence of random graphs $\{G_n\}$, we say that $G_n$ has a property P asymptotically almost surely (a.a.s.) if $\p(G_n\hbox{ has P})\to 1$ as $n\to \infty$. A graph property P is called {\it monotonic} if, whenever $H=(V, E')$, $E'\subset E$, has P and $H$ is a subgraph of $G$, then $G$ also has P; that is, the property is preserved if additional edges are added to the graph. 

Let $G_1$ and $G_2$ be random graphs with $V(G_1)=V(G_2)=V$. We say that $G_2$ {\it dominates} $G_1$ if for all $u, v\in V$, $\p_{G_2}(u\sim v)\geq \p_{G_1}(u\sim v)$. We note the following lemma for monotone graph properties, which is a standard exercise in random graphs (see, for example, \cite{friedgut1996every, erdHos1960evolution}):

\begin{lemma}\label{dominated graphs and connectivity}
Let $G_{1}, G_{2}$ be random graphs with the same vertex set such that $G_2$ dominates $G_1$. If P is a monotone graph property, and $G_1$ has P a.a.s., then $G_2$ also has P a.a.s..
\end{lemma}

We denote by $G(m,p)$ the Erd\H{o}s-R\'{e}nyi graph with $m$ vertices, such that the probability that any two vertices are adjacent is $p$. We recall the following classical result on connectivity in $G(m, p)$ (see, for example, \cite{AlonandSpencer}).

\begin{theorem}\label{T:Gnpconn}
Let $G=G(m, p)$. If there exists $\eps>0$ such that $p>\frac{(1+\eps)\ln m}{m}$, then $G$ is connected a.a.s.. On the other hand, if there exists $\eps>0$ such that $p<\frac{(1-\eps)\ln m}{m}$, then $G$ is disconnected a.a.s..
\end{theorem}

For our purposes, the monotone graph property of greatest interest is the property of a graph being connected a.a.s.. By combining Lemma \ref{dominated graphs and connectivity} with Theorem \ref{T:Gnpconn}, we have the following immediate result:

\begin{lemma}\label{T:domconn}
Let $G$ be a random graph on $m$ vertices. If there exists $\eps>0$ such that for all $u, v\in V(G)$, $\p(u\sim v)> \frac{(1+\eps)\ln m}{m}$, then $G$ is connected a.a.s..
\end{lemma}

As above, for a metric space $(X, d)$ and a function $f:\R^+\to [0, 1]$, let $G=G(X, f)$ be the random graph with $V(G)=X$ and $\p(u\sim v) = f(d(u, v))$ for all $u, v\in V$. We refer to this graph as a {\it random distance graph}. We shall use the notation $\Z^r$ and $L_n^r$ as defined in the introduction.

As for probabilistic tools, we shall require nothing more complex than Markov's Inequality, included here for completeness.

\begin{theorem}[Markov's Inequality]
Let $X$ be a random variable with $X\geq 0$. Then for all $a>0$, we have
\[\p(X>a)\leq \frac{\e{X}}{a}.\]
\end{theorem}

Throughout, we shall use the standard asymptotic notations of $\bigOh{\cdot}$, $\lilOh{\cdot}$, $\ll$, $\gg$, etc. We refer the reader to \cite{cormen2009introduction} for a full formal definition of these notations. Asymptotics will always be considered with respect to $n$; for example, if we write $f(n, r)=\bigOh{g(n, r)}$, it is implied that $r$ is to be held constant and the limit to be considered as $n\to\infty$.

\section{Connectivity in a traditional Waxman model}\label{S:Waxman}

Recall the definition of $W(X, n, f)$ as given in the introduction to be a Waxman graph over a metric space $(X, d)$, where vertices are embedded randomly in $X$ according to some probability distribution $\mu$, and $\p(u\sim v)=f(u, v)$. Waxman graphs have been used to generate models of random networks for modeling systems such as the Internet graph and various biological networks \cite{faloutsos1999power,calvert1997modeling}. The most traditional version of a Waxman graph is formed by taking the underlying metric space as $X=[0,1]^r$, a subset of $\R^r$ under the $\ell^2$ metric, and the distribution $\mu$ to be uniform over $X$. The function $f_n$ is typically chosen to be constant with respect to $n$, with $f_n(d)=f(d)=\alpha e^{-\frac{d}{\beta}}$ with $\alpha, \beta \in (0,1]$. It is commonly known, though no formal proof has been presented to our knowledge, that the graph $W(X, n, f)$ is connected a.a.s.. In fact, we can extend this even to the case that $f_n$ is not constant with respect to $n$, as follows.

\begin{theorem}\label{fixedms}
Let $X$ be a connected metric space with finite diameter $\rho_X$ and let $\mu$ be a probability distribution over $X$. Let $G=W(X, n, f_n)$. If there exists $\eps>0$ such that the functions $f_n$ satisfy the condition
\[
\frac{n}{\log n} f_n(\rho_X) > 1+\eps
\]
for $n$ sufficiently large, then the graph $G$ is connected a.a.s.
\end{theorem}

\begin{proof}
For all $x,y$ vertices in $X$, we see that $\mathbb{P}(x\sim y) > f_n(\rho_X)$. By hypothesis, $\frac{n}{\log n} f_n(\rho_X) > 1+\eps$ for $n$ sufficiently large, so by Lemma \ref{T:domconn}, $G$ is connected a.a.s.
\end{proof}

We find that the above theorem applies to traditional Waxman models by fixing $f_n$, as we thus have $f_n(\rho_X)=f(\rho_X)$ is constant. The application of the theorem leads to no further consequence in traditional Waxman models, so we depart from the model for the purposes of this paper. Indeed, overall the traditional Waxman model becomes locally quite dense over time, and it is for this reason that we depart from this model, and allow the functions $f_n$ and the metric spaces themselves to change with $n$.

\section{Connectivity in $G(X_n, f_n)$}
In this section, we prove Theorems \ref{T:conn}, \ref{T:disconn}, and \ref{T:main}, regarding the a.a.s.\ connectivity of random distance graphs having as their vertex sets nested connected finite metric spaces $(X_n,\rho)$ with finite diameter. 

We begin with the two most general theorems, namely, Theorems \ref{T:conn} and \ref{T:disconn}. We note that Theorem \ref{T:conn} relies almost entirely on Lemma \ref{T:domconn}. Throughout this section, we assume that $f_n$ is a monotonically decreasing function for each $n$.

\begin{proof}[Proof of Theorem \ref{T:conn}]
Let $H=G(X_n,f_n(\diam X_n))$. Then as $f_n$ is monotonically decreasing, $H$ is an Erd\H{o}s-R\'enyi graph that dominates $G_n$ and is a.a.s. connected by Theorem \ref{T:Gnpconn}. Therefore $G_n$ is a.a.s.\ connected by Lemma \ref{T:domconn}.
\end{proof}

We now turn to the proof of Theorem \ref{T:disconn}. We here use a simplified version of the proof that will be used in the case that $X_n=L^r_n$.

\begin{proof}[Proof of Theorem \ref{T:disconn}]
Let $v\in V(G_n)$. We note that as $|\{u\in V(G_n)\ | \ \rho(u, v)=d\}|\leq a_n(d)$ and $f_n(d)\leq\frac{|X_n|^{-\alpha}}{a_n(d)\rho_n}$ for all $d>0$, we thus have
\[
\mathbb{E}[\deg_{G_n}(v)]\leq \sum_{d=1}^{\rho_n}a_n(d)f_n(d)\leq\sum_{d=1}^{\rho_n}\frac{|X_n|^\alpha}{\rho_n}=|X_n|^{-\alpha}.
\]
By Markov's inequality,
\[
\mathbb{P}(v\,\textrm{not isolated in}\, G_n)=\mathbb{P}\left(\deg_{G_n} (v)>\frac{1}{2}\right)\leq 2|X_n|^{-\alpha}
\]
Therefore, the expected number of nonisolated vertices is at most $2|X_n|^{1-\alpha}$. Let $\delta\in(0,1)$ with $1-\delta<\alpha$. By Markov's inequality again,
\[
\mathbb{P}(G_n \,\textrm{has at least}\, |X_n|^\delta\,\textrm{nonisolated vertices)}\leq 2|X_n|^{1-\alpha-\delta}.
\]
Therefore, with probability at least $1-2|X_n|^{1-\alpha-\delta}=1-\lilOh{1}$, $G_n$ has at least $|X_n|-|X_n|^\delta=|X_n|(1-|X_n|^{\delta-1})=|X_n|(1-\lilOh{1})$ isolated vertices.

\end{proof}

As noted in the introduction, there is some difference in the size of the two bounds on $f_n$ in these two theorems. It seems likely that tighter restrictions on $a_n(d)$ could improve the second theorem substantially, if one controls the type of metric space permitted. We note also that straightforward generalizations of these theorems can be derived in the case that $X_n$ is not a finite metric space, but instead we take a Waxman-like approach, and choose finitely many vertices from a single metric space with a finite diameter. 

\subsection{Proof of Theorem \ref{T:main}}\label{S:mainthm}

In this section we focus our analysis on Theorem \ref{T:main}; that is, the case that $X_n=L_n^r$, the $r$-dimensional integer lattice of width $n$ in each dimension under the $\ell^1$ metric, which we shall denote by $\rho$, and $f_n(d) = \frac{1}{n^\beta d}$ for some $\beta>0$. As the proofs of the two parts of the theorem are substantially different in character, we write them as two separate theorems below. We begin with the first statement. Throughout this section, we shall use the following notation.

 Let $\Z^{r}$ denote the $r$-dimensional integer lattice. Write $L_{n}^{r}=\{a \in \Z^r\ | \ 0\leq a_i\leq n-1\hbox{ for all } i\}\subset \Z^r$, the $n\times n\times\dots\times n$ integer lattice in $r$ dimensions. We call $L_{n}^{r}$ the $r-$dimensional lattice of size $n$, and when context makes $n,r$ clear, we write $L=L_{n}^{r}$ and $\La=\Z^r$. Our primary focus will be on the graph $G(L, f_n)$, where $f_n(d) = \frac{1}{n^\beta d}$. We begin with the first statement in Theorem \ref{T:main}, restated below for convenience, whose proof mirrors that of Theorem \ref{T:conn}.

\setcounter{theorem}{3}
\begin{theorem}[Part 1]
Let $L=L_n^r$, and $G=G(L, f_n)$, where $\beta<r-1$. Then $G$ is connected a.a.s..
\end{theorem}

\begin{proof}
Let $u, v\in L$. Note by definition that $\rho(u, v)\leq r(n-1)<rn$, and hence $\p(u\sim v)\geq f_n(rn) = \frac{1}{rn^{1+\beta}}=:p$.

Moreover, $G$ has $n^r$ vertices. Note that $n^rp=n^{r}\frac{1}{rn^{1+\beta}}=\frac{1}{r}n^{r-1-\beta}\gg\log(n^{r})$ when $\beta<r-1$. But then by Lemma \ref{T:domconn}, we immediately have that $G$ is connected a.a.s..
\end{proof}

We now turn our attention to the proof of the second half of Theorem \ref{T:main}. To prove the second half, we shall view $G(L, f_n)$ as a subgraph of the infinite graph $G(\La, g_n)$, where $g_n(d)=f_n(d)$ if $d\leq r(n-1)$ and $0$ otherwise. Note that it is sufficient to prove that $G$ has $n^r(1-\lilOh{1})$ isolated vertices whenever $\beta>r-1$. To do so, we shall view $G(L, f_n)$ as a subgraph of the infinite graph $G(\La, g_n)$, where \begin{equation}\label{E:gdef}g_n(d)=\left\{\begin{array}{ll}f_n(d) & \hbox{ if }d\leq r(n-1)\\0&\hbox{ otherwise}\end{array}\right. .\end{equation} The structure of the proof is similar to that of the proof of Theorem \ref{T:disconn}, however we shall be able to develop much more precise estimates on $a_n(d)$ in this case.

To begin, note $G(\La, g_n)$ is vertex transitive, and hence the expected degree is the same for every vertex. Fix a vertex in $v\in L^r_n$, and define $a^{(v)}_{r}(d)$ to be the number of vertices $u$ in $L^r_n$ such that $\rho(u, v)=d$. For $d\leq r(n-1)$, let $a_{r}(d)$ denote the number of vertices $u$ in $\Z^r$ with $\rho(u, v)=d$, and define $a_{r}(d)=0$ for $d>r(n-1)$. By definition, we have $a^{(v)}_r(d)\leq a_r(d)$, and hence note the following simple observations:
\begin{equation}\label{E:degininfL}\e{\deg_{\Z^r}(v)} = \sum_{d=0}^{r(n-1)} a_r(d)g_n(d),\end{equation}
 and\begin{equation}\label{E:deginL} \e{\deg_{L_n^r}(v)} = \sum_{d=0}^{r(n-1)} a^{(v)}_r(d)f_n(d)\leq \e{\deg_{\Z^r}(v)}.\end{equation}

As $a_r(d)$ is independent of the chosen vertex $v$, we thus have a uniform bound on the expected degree of any vertex in $G(L_n^r, f_n)$. In order to make this bound useful, we shall use the following recursive formula for $a_r(d)$. We note that in this formula, we shall take $a_r(0)=1$, as a vertex has exactly one vertex at distance 0 to it, namely, itself.

\begin{lemma}\label{L:ard}
For $d\leq 2n-2$, $a_{2}(d)=4d$, and for $d\leq r(n-1)$, 
\[
a_{r+1}(d)=2\left(\sum_{k=0}^{d-1}a_{r}(k)\right)+a_{r}(d).
\]
\end{lemma}

\begin{proof} As noted above, $a_r(d)$ is independent of the vertex $v$; let us suppose that $v=\mathbf{0}$. We take $(b_1, b_2,\ldots, b_{r+1})$ to be a point in $\Z^{r+1}$. Let us consider, then 

\begin{eqnarray*}
a_{r+1}(d) &=& \left| \left\{(b_1, b_2, \dots, b_{r+1})\ \vert \ \sum |b_i| = d\right\} \right|\\
& = & \sum_{k=0}^d \left| \left\{(b_1, b_2, \dots, b_{r+1})\ | \ \sum |b_i|=d \hbox{ and }|b_{r+1}|=k\right\}\right|.
\end{eqnarray*}

That is to say, we can view $\Z^{r+1}$ as an infinite stack of copies of $\Z^r$, arrayed along the $(r+1)^{\textrm {st}}$ axis. To calculate $a_{r+1}(d)$, we then simply add up the values of $a_{r}(k)$ contributed from each copy. Note that if $|a_{r+1}|=k\neq 0$, we have
\[\left|\left\{(b_1, b_2, \dots, b_{r+1})\ | \ \sum |b_i|=d \hbox{ and }|b_{r+1}|=k\right\}\right|= 2\left|\left\{(b_1, b_2, \dots, b_r)\ | \ \sum |b_i|=d-k\right\}\right|,\]
where the 2 is to accommodate the duplication for $a_{r+1}=\pm k$. The case that $k=0$ is identical, without the factor of two. 

Together with the above, we thus obtain 
\begin{eqnarray*}
a_{r+1}(d)  & = & \sum_{k=0}^d \left| \left\{(b_1, b_2, \dots, b_{r+1})\ | \ \sum |b_i|=d \hbox{ and }|b_{r+1}|=k\right\}\right|\\
& = & \sum_{k=1}^d 2\left|\left\{(b_1, b_2, \dots, b_r)\ | \ \sum |b_i|=d-k\right\}\right| + \left|\left\{(b_1, b_2, \dots, b_r)\ | \ \sum |b_i|=d\right\}\right|\\
& = & \sum_{k=1}^d 2a_r(d-k) + a_r(d).
\end{eqnarray*}

Reindexing this sum yields the stated result.

For the case that $r=2$, note that we can apply the above calculation to obtain that for $d\leq 2(n-1)=2n-2$,
\[ a_2(d) = 2\sum_{k=0}^{d-1} a_1(k) + a_1(d).\]
Note that in dimension 1, there are precisely two vertices at distance $k$ for any positive $k$, and one vertex at distance 0. Hence, we have
\[a_2(d) = 2(1 + 2(d-1))+2 = 4d.\]

\end{proof}

\begin{lemma}\label{expecteddegree}
Let $H_n^{r}=G(\Z^{r},g_n)$, where $g_n$ is as in Equation \eqref{E:gdef}. Fix $v\in V({H_n^r})$. Then for all $r\geq 2$,
\[
\e{\deg_{H_n^{r}}(v)} = \bigOh{n^{r-1-\beta}}.
\]
\end{lemma}

\begin{proof}
We work by induction on $r$. For simplicity of notation, we write $H_n^r$ as $H$ or $H^r$ when $n$ is clear. 

First, suppose $r=2$. By Lemma \ref{L:ard} and Equation \eqref{E:degininfL}, we thus have

 \[ \e{\deg_{H}(v)} = \sum_{d=1}^{2(n-1)}a_2(d)g_n(d) = \sum_{d=1}^{2(n-1)}4d\left(\frac{1}{n^\beta d}\right) = 4(2n-2)\frac{1}{n^\beta}<8n^{1-\beta}.\]
 
Hence, the case that $r=2$ is established. Now, for induction, suppose that the result holds for $r$. Note by Lemma \ref{L:ard} that for any vertex $v\in V(H^{r+1})$, we have

\begin{align*}
\e{\deg_{H^{r+1}}(v)} &= \sum_{d=1}^{(r+1)(n-1)} a_{r+1}(d)g_n(d)\\
 &= \frac{1}{n^\beta} \sum_{d=1}^{(r+1)(n-1)}  \frac{2\left( \sum_{k=0}^{d-1} a_r(k) \right) + a_r(d)}{d}\\
&= \frac{1}{n^\beta} \sum_{d=1}^{(r+1)(n-1)} \frac{a_r(d)}{d} + \frac{1}{n^\beta} \sum_{d=1}^{(r+1)(n-1)} \frac{2}{d} \sum_{k=0}^{d-1} a_r(k).
\end{align*}

For the first term, notice that $a_r(d)=0$ by definition if $d>r(n-1)$, and hence
\begin{equation}\label{E:firstterm} \frac{1}{n^\beta}\sum_{d=1}^{(r+1)(n-1)} \frac{a_r(d)}{d} =  \frac{1}{n^\beta} \sum_{d=1}^{r(n-1)} \frac{a_r(d)}{d} = \e{\deg_{H^r}(v)}= \bigOh{n^{r-1-\beta}},  \end{equation}
by the inductive hypothesis.

For the second term, we may change the order of summation to obtain
\[ \frac{1}{n^\beta} \sum_{d=1}^{(r+1)(n-1)} \frac{2}{d} \sum_{k=0}^{d-1} a_r(k) =  \frac{1}{n^\beta}\left(2\sum_{k=1}^{(r+1)(n-2)\frac{1}[d}+ \sum_{k=1}^{(r+1)(n-1)-1} \frac{a_r(k)}{k} \sum_{d=k+1}^{(r+1)(n-1)} \frac{2k}{d}\right).\]

The first term corresponds to the case that $k=0$, the second to all other values of $k$. Note that for the first term, we have 
\begin{equation*}
\frac{1}{n^\beta}\sum_{d=1}^{(r+1)(n-2)}\frac{2}{d}\leq \frac{1}{n^\beta}2(r+1)(n-2) = \bigOh{n^{1-\beta}} = \bigOh{n^{r-1-\beta}},
\end{equation*}
since $r\geq 2$.

For the second term, we have $\frac{2k}{d}\leq \frac{2k}{k+1}\leq 2$ for all $k>0$. Further, we can apply the property that $a_r(k)=0$ if $k>r(n-1)$, and we thus have

\begin{eqnarray} \nonumber\frac{1}{n^\beta} \sum_{k=1}^{(r+1)(n-1)-1} \frac{a_r(k)}{k} \sum_{d=k+1}^{(r+1)(n-1)} \frac{2k}{d} &\leq& \frac{1}{n^\beta}\sum_{k=1}^{r(n-1)}\frac{a_r(k)}{k}\sum_{d={k+1}}^{(r+1)(n-1)}2\\
\nonumber & \leq & 2((r+1)(n-1)-1)\frac{1}{n^\beta}\sum_{k=1}^{r(n-1)}\frac{a_r(k)}{k}\\
\nonumber & = & 2(n(r+1)-(r+2))\bigOh{n^{r-1-\beta}}\\
\label{E:secondterm}& = & \bigOh{n^{r-\beta}}.
\end{eqnarray}

Taking Equations \eqref{E:firstterm} and \eqref{E:secondterm} together, we obtain
\[\e{\deg_{H^{r+1}}(v)} = \bigOh{n^{r-1-\beta}} +\bigOh{n^{r-\beta}} =\bigOh{n^{r-\beta}},\] as desired.
\end{proof}

\setcounter{theorem}{3}
\begin{theorem}[Part 2]
Let $G=G(L_n^r,f_n)$ with $r$ fixed and $f_n(d)=\frac{1}{n^{\beta}d}$. If $\beta>r-1$, then there exists $\eps>0$ such that with probability at least $1- n^{-\eps}$, $G$ has $n^r(1-\lilOh{1})$ isolated vertices.
\end{theorem}

\setcounter{theorem}{7}

\begin{proof}

Let $G=G(L_n^r, f_n)$, where $\beta>r-1$ and $r\geq 2$, and let $H=G(\Z^r, g_n)$. By Lemma \ref{expecteddegree}, we thus have that there exists some constant $c$ such that, for any vertex $v\in V(G)$,
\[\e{\deg_G(v)} \leq \e{\deg_H(v)}\leq cn^{r-1-\beta}.\]

Thus, by Markov's Inequality, we have that
\[\p(v\hbox{ is not isolated in }G ) =\p\left(\deg(v)>\frac{1}{2}\right)\leq 2cn^{r-1-\beta}.\] Hence, the expected number of nonisolated vertices in $G$ is at most $n^r(2cn^{r-1-\beta})=2cn^{2r-1-\beta}$. By Markov's inequality again, for any $\eps>0$, we have
\[\p\left(G\hbox{ has at least }2cn^{2r-1-\beta+\eps}\hbox{ nonisolated vertices}\right)\leq n^{-\eps}=\lilOh{1}.\]

Take $\eps=\frac{-r+\beta+1}{2}>0$. Note then that as $r-\beta-1<0$, that $r-\beta-1+\eps<0$. Thus, we have that with probability at least $1-n^{-\eps}=1-\lilOh{1}$, $G$ has at least 
\[n^r-2cn^{2r-1-\beta+\eps}=n^r\left(1-2cn^{r-1-\beta+\frac{r-\beta-1}{2}}\right) = n^r(1-\lilOh{1})\]
isolated vertices, as desired.

\end{proof}

\section{Expected degree in $G(L_n^2, f_n)$}\label{S:degreetwo}
For completeness we include an exact analysis on the expected degree of a vertex in $G=G(\mathbb{Z}^2,f_n)$ and determine $a_2^{(v)}(d)$ exactly. We do so by explicitly counting the number of vertices at distance $d$ from a given vertex $v=(d_1, d_2)\in L_n^2$. 

Throughout this section, we shall keep $v$ fixed as $(d_1, d_2)$, and hence we will suppress the superscript $(v)$ and simply write $a_2(d)$ in place of $a_2^{(v)}(d)$. Likewise, we shall restrict to working in the $n\times n$ integer lattice, which we shall denote simply by $L$.

Fix a distance $d\leq 2(n-1)$. Recall that from Lemma \ref{L:ard}, in $\Z^2$, there are $4d$ vertices at distance $d$ from $v$. Hence, we need only determine how many of these vertices are in fact members of $L$.

Let $S_d(v)$ be a square of side length $2d$ centered at $v$. We note that not all vertices in $S_d(v)$ will be within distance $d$ of $v$; however, all vertices at distance precisely $d$ from $v$ are contained in $S_d(v)$. By considering the corners of the square $S_d(v)$, we thus have that if \begin{equation}\label{Conditions}d_1-d\geq 0, d_2-d\geq 0, d_1+d\leq n-1, \hbox{ and }d_2+d\leq n-1,\end{equation} then $a_2(d)=4d$.

If these four conditions are not all met, then we have that $S_d(v) \cap (\Z^2\ba L)\neq \emptyset$; our main task then is to count how many vertices at distance $d$ from $v$ lie outside of $L$.

First, consider the case that only one of these inequalities fails; without loss of generality, suppose that $d_1-d<0$. This case is illustrated in Figure \ref{firstcase}. Let $c=(-1, d_2)$, and note that $d(c, v)=d_1+1$. Note that if $u=(x, y)$ is a vertex at distance $d$ from $v$, with $u\notin L$, then we have $x<0$, so that $u$ is obtained from $c$ by taking $k$ steps left and $d-k-d_1-1$ steps either up or down. Hence, there will be $1+2(d-d_1-2)$ such vertices, where we obtain 1 vertex for the case that $k=d-d_1-1$ and 2 vertices (corresponding to steps up or down) in all other cases.

\begin{figure}[htp]
\includegraphics[width=.4\textwidth]{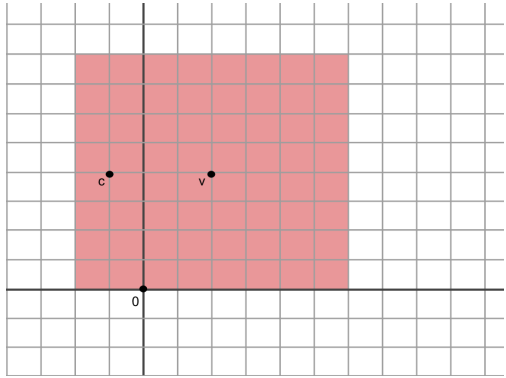}
\caption{An illustration of the case that $d_1-d<0$, but all other conditions in \eqref{Conditions} are met. Here, the shaded region represents $S_d(v)$, and we see that $S_d(v)$ intersects $\Z^2\ba L$ only on one of the four sides.
}\label{firstcase}

\end{figure}

Hence, in the case that $d>d_1$, and all other conditions of \eqref{Conditions} are met, we have that $a_2(d) = 4d-2(d-d_1-2)-1=2d+2d_1+3$.

In all other cases, we shall apply the same technique. We thus need only determine the number of vertices that are double counted by this technique. Without loss of generality, we shall consider this double count only for the case that $d>d_1$ and $d>d_2$; all other cases will be symmetric. This situation is illustrated in Figure \ref{secondcase}.

\begin{figure}[htp]
\label{secondcase}
\includegraphics[width=.4\textwidth]{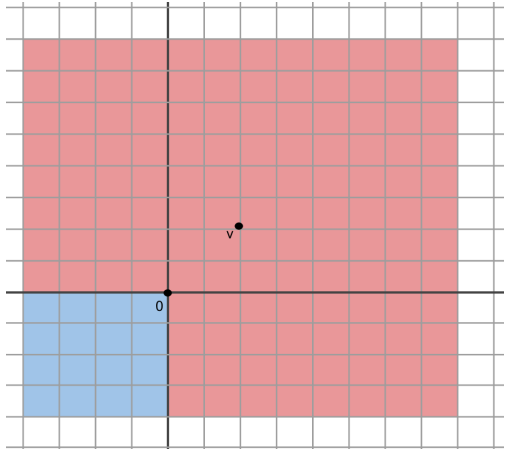}
\caption{An illustration of the case that $d_1-d<0$ and $d_2-d<0$, but all other conditions in \eqref{Conditions} are met. Here, the shaded regions represent $S_d(v)$, and we see that $S_d(v)$ intersects $\Z^2\ba L$ on two of the four sides. The blue shaded region represents vertices that will be double counted by the technique used in the first case.
}\label{secondcase}
\end{figure}

Notice that here we need to count the number of vertices $u=(x, y)$ such that $d(u, v)=d$ and $x<0$, $y<0$. Notice that $d(0, v)=d_1+d_2$, and hence any such vertex $u$ has $d(0, u)=d-d_1-d_2$ (we note also here that if $d_1+d_2\geq d$, there is nothing to count). Note that by symmetry, exactly $\frac{1}{4}$ of the vertices at this distance to $u$, excluding the axes, shall occur in the blue shaded region shown in Figure \ref{secondcase}. Excluding vertices on the axes, there are $4(d-d_1-d_2-1)$ such vertices; hence the number of such vertices with both $x<0$ and $y<0$ is precisely $d-d_1-d_2-1$.

Combining these results and applying symmetry, we thus obtain the following theorem.

\begin{theorem}
Let $\delta_x = 1$ if $x<0$ and $0$ otherwise. Then
\begin{eqnarray*}
a_2(d)&=&4d-\delta_{d_1-d}(2(d-d_1)-3) - \delta_{d_2-d}(2(d-d_2)-3) - \delta_{n-1-d_1-d}(2(d_1+d-n+1)-3)\\ &&- \delta_{n-1-d_2-d}(2(d_2+d-n+1)-3) + \delta_{d_1+d_2-d}(d-d_1-d_2-1) + \delta_{d_1+n-1-d_2-d}(d-d_1-(n-1)+d_2-1)\\
&&  + \delta_{d_2+n-1-d_1-d}(d-d_2-(n-1)+d_1-1)+ \delta_{2(n-1)-d_1-d_2-d}(d-2(n-1)+d_1+d_2-1) 
\end{eqnarray*}

\end{theorem}

\section{Conclusions}

Although the traditional Waxman graph has been widely used in some areas of social science, its mathematical features have to date not been studied in detail. Here, we find that as a network model, the Waxman graph has some deficiencies, particularly in its connectivity structure, and hence it may be more reasonable, and perhaps not more difficult, to replace this model with a model as proposed herein.

In addition, further study on the structure of a random distance graph as described herein, for which vertices are chosen randomly from the underlying metric space $X$, would be an interesting future direction for this research. Moreover, in the specific case studied in Theorem \ref{T:main}, it would be interesting to determine if a more nuanced choice of $\beta$, perhaps dependent on $n$, might yield the typical ``double-jump'' behavior for random graphs.

\section{Acknowledgements}
The authors are grateful to Ryan Dingman for his contribution to the initial stages of development of this project and its results, and to Toby Johnson for some useful discussion on the traditional Waxman model.

\bibliographystyle{siam}  
 \bibliography{bib_items}

\end{document}